\documentclass{amsart}
\usepackage{todonotes}
\usepackage{xypic}
\usepackage{graphicx, amsmath, amssymb, amsthm}
\usepackage{tikz-cd}
\usepackage{hyperref}

\newcommand{\support}[1]{The author was supported by {#1}.}

\newcommand{\NSFThree}{NSF Grant DMS-1509652}

\newcommand{\MAHAddress}{University of California Los Angeles, Los Angeles, CA 90095}
\newcommand{\MAHemail}{\tt{mikehill@math.ucla.edu}}

\newcommand{\Z}{{\mathbb  Z}}

\newcommand{\F}{{\mathbb F}}

\DeclareMathOperator{\Hom}{Hom}

\DeclareMathOperator{\Map}{Map}

\DeclareMathOperator{\Stab}{Stab}

\DeclareMathOperator{\Sing}{Sing}
\DeclareMathOperator*{\colim}{colim}
\DeclareMathOperator{\Aut}{Aut}
\DeclareMathOperator{\Conf}{Conf}

\newcommand{\res}{res}

\newcommand{\m}[1]{{\protect\underline{#1}}}

\newcommand{\mM}{\m{M}}

\newcommand{\mR}{\m{R}}
\newcommand{\mB}{\m{B}}

\newcommand{\mH}{\m{H}}

\newcommand{\mA}{\m{A}}

\newcommand{\cc}[1]{\mathcal #1}

\newcommand{\ccD}{\cc{D}}
\newcommand{\cF}{\cc{F}}
\newcommand{\cO}{\cc{O}}

\newcommand{\Sp}{\mathcal Sp}
\newcommand{\Emb}{\textnormal{Emb}}

\newcommand{\Ninfty}{N_\infty}

\newcommand{\Top}{\mathcal Top}

\newcommand{\Orb}{\mathcal Orb}

\mathchardef\mhyphen=45

\newcommand{\EM}{Eilenberg-Mac~Lane}

\numberwithin{equation}{section}

\newtheorem{theorem}{Theorem}[section]
\newtheorem{lemma}[theorem]{Lemma}
\newtheorem{corollary}[theorem]{Corollary}

\newtheorem{proposition}[theorem]{Proposition}

\newtheorem*{theorem*}{Theorem}
\newtheorem*{proposition*}{Proposition}

\theoremstyle{remark}
\newtheorem{remark}[theorem]{Remark}

\theoremstyle{definition}

\newtheorem{definition}[theorem]{Definition}

\newcommand{\defemph}[1]{\textbf{#1}}

\title[Signed Loop Spaces]{On the algebras over equivariant little disks}

\author{Michael A.~Hill}
\thanks{\support{\NSFThree}}
\address{\MAHAddress}
\email{\MAHemail}

\begin{document}

\begin{abstract}
We describe the structure present in algebras over the little disks operads for various representations of a finite group $G$, including those that are not necessarily universe or that do not contain trivial summands. We then spell out in more detail what happens for $G=C_{2}$, describing the structure on algebras over the little disks operad for the sign representation. Here we can also describe the resulting structure in Bredon homology. Finally, we produce a stable splitting of coinduced spaces analogous to the stable splitting of the product, and we use this to determine the homology of the signed James construction.
\end{abstract}

\maketitle

\section{Twisted equivariant ``monoids''}
We begin with a purely algebra observation. Consider the free associative monoid on the set $C_2=\{x, \bar{x}\}$. The fixed points of this are the one point set consisting of the identity element, and elements which we expect from the commutative case, like the norm classes $x\bar{x}$ are not fixed here. If we enforce commutativity, then these become fixed, but as homotopy theorists, we should instead attach cells which connect $x\bar{x}$ to $\bar{x}x$, building a $C_2$-CW complex. Unfortunately, since $x\bar{x}$ and $\bar{x}x$ form a free orbit, these new cells will also be $C_2$-free. Continuing to make the multiplication more highly commutative will never add fixed cells, so we never produce new fixed points. This is reflecting a classical observation.

\begin{proposition}
If $\cO^{tr}$ is an $E_\infty$ operad with trivial $G$-action, and $X$ is a free $G$-space, then the fixed points of the free algebra are trivial:
\[
\big(\mathbb P_{\cO^{tr}}(X)\big)^G\simeq \ast.
\]

More generally, if $V$ is a countable dimensional vector space with trivial $G$-action, and if $\ccD(V)$ is the little disks operad for $V$, then the free algebra on $X$ also has trivial fixed points:
\[
\big(\mathbb P_{\ccD(V)}(X)\big)^G\simeq\ast.
\]
\end{proposition}

This is in stark contrast with the ``genuine'' commutative case, where the tom Dieck splitting shows that the putative norm classes do give rise to new fixed points.

\begin{proposition}
If $\cO$ is a $G$-$E_\infty$ operad, and if $X$ is a free $G$-space, then we have
\[
\big(\mathbb P_{\cO}(X)\big)^G\simeq \mathbb P_{\cO^{tr}}(X/G).
\]
\end{proposition}
The obvious map
\[
i_e^\ast\mathbb P_{\cO}(X)\simeq \mathbb P_{\cO^{tr}}(i_e^\ast X)\to \mathbb P_{\cO^{tr}}(X/G)\hookrightarrow \big(\mathbb P_{\cO}(X)\big)^G
\] 
is the operadic norm from the underlying space to the fixed points. Adams' original discussion of the construction of the transfer in the $G$-equivariant Spanier-Whitehead category shows that we have a similar norm map for free algebras over the little disks algebra for any orthogonal $G$-representation in which $G$ equivariantly embeds. 

The primary goal of this paper is to study these exact phenomena by carefully unpacking some of the structure that we see in these algebras over little disks operads for (possibly finite dimensional) orthogonal $G$-representations. Since a trivial $E_\infty$ operad and a $G$-$E_\infty$ operad are both little disks operads, we will recover our classical understanding. More generally, a little disks operad for a $G$-universe $U$ is the prototype of an $\Ninfty$ operad, the definition of which and the basic properties thereof were described by Blumberg and the author in \cite{BHNinfty}. In particular, the $\Ninfty$ operads behave very much like ordinary $E_\infty$ operads plus operations which coherently encode norm maps. A natural question then, is what operads deserve to be called $\mathcal N_k$ operads for finite $k$. Exploring the examples for little disks gives desiderata for defining the $\mathcal N_{k}$, but we leave this for later work.

Work of Rourke and Sanderson \cite{RouSand}, Segal \cite{SegalConf}, and Hauschild \cite{Hausch}, equivariantizing foundational work of May \cite{MayGeom}, connects this general analysis to the study of loop spaces. They show that if $\cO$ is the little disks operad for a representation $V$, then the natural map
\[
\mathbb P_{\cO}(X)\to \Omega^{V}\Sigma^{V}X
\]
is a group completion. Thus our operadic analysis describes geometric content for these generalized loop spaces. In particular, it provides a basic step for describing the operations on $V$-fold loop spaces, in the vein of Cohen \cite{CohenHstar}.

A secondary goal of this paper is to build on the connection to these equivariant loop spaces, collecting some results about signed-loop spaces for the group $C_2$ which have become folklore. We will study the basic properties of signed loop spaces. We then turn to cohomology, looking at the cohomology of a coinduced space and the resulting structure on the homology of a signed loop space. We next study the signed version of the James construction on a $C_2$-space, showing that after suspension by the regular representation it splits. As an appendix, we include a result of independent interest: that coinduction for the group $C_{2}$ splits after a single signed suspension.

\subsection*{Notation and conventions}
In all that follows, $G$ will denote a finite group, and $H$ will usually denote a subgroup thereof. Letters around $X$ in the alphabet will denote $G$-spaces; letters around $T$ in the alphabet will denote $G$-sets; and letters around $V$ in the alphabet will denote $G$-representations. 

If $X$ is a $G$-space and $x\in X$ is any point, then $\Stab(x)$ will denote the stabilizer subgroup of $x$ in $G$. If $H\subset G$, then $W_{G}(H)$ will denote the Weyl group of $H$ in $G$.

In the second half of the paper, we will begin doing equivariant algebra, and we will work often with Mackey functors for the cyclic group of order $2$, $C_{2}$. The Mackey functor versions of standard classical constructions like homotopy or homology groups will be denoted with an underline to the classical symbol. For Mackey functors, we will follow Lewis' notation, stacking the values at the two orbits and indicating the structure maps:
\[
\xymatrix{
{\m{M}(C_{2}/C_{2})}
	\ar@(l,l)[d]_{\res}
	\\
{\m{M}(C_{2}/e)}
	\ar@(r,r)[u]_{tr}
	&
}
\]

Finally, for gradings, we follow the increasingly standard wild-card notation: $\ast$ will be reserved for an arbitary element of $\mathbb Z$, while $\star$ will denote an element of $RO(G)$.

\subsection*{Acknowledgements}

We thank Mark Behrens and Dylan Wilson for several very inspiring conversations early on in this project and for reading an early preprint. We also think Kirsten Wickelgren, Ugur Yigit, and Doug Ravenel for several helpful conversations about the James construction and signed loops. We also thank Andrew Blumberg and Tyler Lawson for help clarifying several operadic points. 

\section{Algebras over equivariant little disks operads}
\subsection{The general case}
We take as our model the analysis of algebras over the little disks for a $G$-universe from \cite[Theorem 4.19]{BHNinfty}. The key idea there was to understand which finite $H$-sets embed in our universe. Essentially everything goes through without change: the only difference is that the spaces parameterizing our generalized multiplications are no longer necessarily contractible. Much of the analysis of the spaces is closely related to results of Rourke and Sanderson \cite{RouSand}; we cast things in a way that most transparently reflects how various norms arise in the algebras.

\subsubsection{Basic Definition}
\begin{definition}
Let $V$ be an orthogonal representation of $G$ (not necessarily finite dimensional, and not necessarily a universe for $G$). Let $W$ be a finite dimensional orthogonal subrepresentation of $V$. A {\defemph{little disk}} in $W$ is a (not necessarily equivariant) affine map $D(V)\to D(V)$. 

Define the space $\cc{D}_W(V)_n$ to be the space of $n$-tuples of nonoverlapping little disks. This is a $G\times\Sigma_n$-space, where $G$ acts by conjugation and $\Sigma_n$ by permuting the coordinates. The operadic structure is by composition.

If $W\subset W'$, then the affine maps defining any of the little disks in $\cc{D}_W(V)_n$ extend to give little disks in $\cc{D}_{W'}(V)_n$, so we define 
\[
\cc{D}(V)_n=\colim_{W} \cc{D}_W(V)_n.
\]
\end{definition}
Of course, if $V$ is finite dimensional, then $\{V\}$ is cofinal in the set of finite dimensional subspaces of $V$, so we can ignore any colimits.

The operadic multiplications and transfers are produced by the fixed points for various subgroups of $G\times\Sigma_n$.

\subsubsection{Structure of graph fixed points}
We begin with a basic observation which follows immediately from the requirement that the little $n$-disks not overlap.
\begin{proposition}
For any orthogonal representation $V$, $\Sigma_n$ acts freely on $\cc{D}(V)_n$.
\end{proposition}
Thus the only subgroups which could have fixed points are the ``graph subgroups'' considered in \cite{BHNinfty}.
\begin{definition}
A {\defemph{graph subgroups}} of $G\times\Sigma_n$ is a subgroup which intersects $\Sigma_n$ trivially.
\end{definition}
The name comes from the fact that if $\Gamma$ is a graph subgroup of $G\times\Sigma_n$, then there is a subgroup $H\subset G$ and a homomorphism $\phi\colon H\to\Sigma_n$ such that $\Gamma$ is the graph of $\phi$. In particular, to each graph subgroup, associate a subgroup $H$ of $G$ and an $H$-set structure on $\{1, \dots, n\}$. If $H$ is a proper subgroup of $G$, then the analysis of the $\Gamma$ fixed points of $\cc{D}(V)$ is most naturally an $H$-equivariant one, and hence is covered by an induction argument on the subgroups. We can therefore restrict attention to $H=G$.

\begin{lemma}\label{lem:GraphFixedPoints}
Let $\Gamma$ be a graph subgroup of $G\times\Sigma_n$ corresponding to a $G$-set $T$. Then the map taking a disk to its center gives a weak equivalence
\[
\cc{D}(V)_n^{\Gamma} \simeq \Emb^{G}(T,V),
\]
where $\Emb^{G}(T,V)$ is the space of $G$-equivariant embeddings of $T$ into $V$. Moreover, this equivalence is 
\[
W_{G\times\Sigma_n}(\Gamma)\cong \Aut^G(T)
\]
equivariant.
\end{lemma}
\begin{proof}
The map sending a little disk to its center establishes a $G\times\Sigma_n$-weak equivalence
\[
\cc{D}(V)_n\simeq \Emb\big(\{1,\dots,n\},V\big),
\]
where $G\times\Sigma_n$ acts on the latter via the $\Sigma_n$-action on the source and the $G$-action on the target. When restricted to $\Gamma$, this becomes the conjugation action on $\Emb(T,V)$. The result is now obvious.
\end{proof}

Since we are considering only finite $G$-sets, these spaces of embeddings are a kind of equivariant configuration space.

\begin{definition}
Let $n\geq 1$, and let $X$ be a $G$-space. Define the {\defemph{configuration space of $n G/H$}} in $X$ to be
\[
\Conf_{nG/H}(X):=\big\{(x_1,\dots,x_n) \mid \forall i\Stab(x_i)=H\text{ and } \forall i\neq j, \forall g\in G, x_i\neq gx_j\big\},
\]
topologized as a subspace of $X^n$.
\end{definition}

Note that by construction
\[
\Conf_{nG/H}(X)\subset (X^H)^n.
\]
The latter has a canonical action of $W_G(H)\wr\Sigma_n$ extending the coordinate actions of $W_G(H)$. By construction, $\Conf_{nG/H}(X)$ is an equivariant subspace. With this observation, the following is immediate.

\begin{proposition}
If $X$ is any space, then we have a homeomorphism
\[
\Emb^G(nG/H,X)\cong \Conf_{nG/H}(X)
\]
given by choosing an ordering of the summands and evaluating at the cosets of the identity for each summand. This is equivariant for
\[
\Aut^G(nG/H)\cong W_G(H)\wr\Sigma_n.
\]
\end{proposition}

When $X$ is a representation, then these are simpler still: they are hyperplane complements.

\begin{definition}
Let $H\subset G$ be a subgroup and $V$ a representation of $G$. 

Let 
\[
V^{\partial H}:=\bigcup_{H\subset K} V^K
\]
be the $H$-relative singular set of $V$, and let 
\[
\mathring{V}^H:=V^H-V^{\partial H}.
\]
\end{definition}


\begin{proposition}
The configuration space of $n G/H$ in a representation $V$ is a hyperplane complement:
\begin{align*}
\Conf_{nG/H}(V)&=\big\{(v_1,\dots,v_n)\in \mathring{V}^H \mid \forall i\neq j, \forall gH\in W_G(H), v_j\neq g v_j\big\} \\
&=(V^H-V^{\partial H})^n - \bigcup_{i\neq j, gH\in W_G(H)} W_{i,j,gH}, 
\end{align*}
where 
\[
W_{i,j,gH}=\big\{(v_1,\dots, v_n) \mid v_i=gv_j\big\}\subset (V^H)^n.
\]
\end{proposition}
\begin{proof}
The subspace $\mathring{V}^H$ is the subspace of $V$ consisting of those points with stabilizer $H$. It itself is a hyperplane complement, so its $n$-fold Cartesian power is. The result follows immediately from the definition.
\end{proof}

Putting this all together, we deduce the graph fixed points for an arbitrary graph subgroup.

\begin{theorem}\label{thm:TopologyofEmbeddings}
Let $T=n_1 G/H_1\amalg\dots \amalg n_k G/H_k$, where if $i\neq j$, then $H_i$ and $H_j$ are not conjugate. Then we have an $\Aut^G(T)$-equivariant homeomorphism
\[
\Emb^G(T,V)\cong \prod_{i=1}^{k} \Conf_{n_iG/H_i}(V).
\]
In particular, this is a hyperplane complement.
\end{theorem}
\begin{proof}
It only remains to prove that the embeddings for non-conjugate stabilizers can never coincide (hence we get a product decomposition). However, this is obvious, since an embedding preserves stabilizers.
\end{proof}

\begin{corollary}\label{cor:EmbModule}
If $T$ is any finite $G$-set such that $T^{G}=\emptyset$, then
\[
\Emb^{G}(T,V)\simeq \Emb^{G}(T\amalg \ast, V).
\]
\end{corollary}
\begin{proof}
The origin is always fixed in $V$, thus $\Conf_{G/G}(V)$ is always non-empty.
\end{proof}
\begin{remark}
If $|T|=n$, then choose an ordering of the points of $T$ and hence a non-equivariant identification of $T$ and $\{1,\dots n\}$. Including $\ast$ as the $(n+1)$st point gives a non-equivariant identification of $T\amalg \ast$ with $\{1,\dots, n+1\}$. In this case, the homotopy equivalence in Corollary~\ref{cor:EmbModule} can be realized by the inclusion of the unit in the $(n+1)$st coordinate.
\end{remark}

\subsubsection{Algebras over little disks}
Lemma~\ref{lem:GraphFixedPoints} is the key result that underpins our analysis.

\begin{theorem}[Compare {\cite[Lemma 6.6]{BHNinfty}}]\label{thm:DiskAlgebras}
Let $X$ be a $\cc{D}(V)$-algebra in spaces, and let $T$ be a finite $H$-set. Then the space $\Emb^H(T,V)$ encodes $H$-equivariant multiplications
\[
\Map(T,X)\to X.
\]
\end{theorem}
\begin{proof}
Let $\Gamma$ be the graph subgroup associated to an ordering of $T$. Lemma~\ref{lem:GraphFixedPoints} shows that we have an equivalence
\[
\Map^{G\times\Sigma_n}\big(G\times\Sigma_n/\Gamma_T, \ccD(V)_n\big)\simeq \Emb^H(T,V),
\]
and that this equivalence is compatible with the ($H$-)Weyl action on the source and the $\Aut^H(T)$-action on the target. If $\Emb^H(T,V)$ is empty, then we have nothing to prove, so assume that this is non-empty. In this case, each point in $\Emb^H(T,V)$ gives a $G$-equivariant map
\[
G\times_H \Map(T,X)\cong G\times_H\big((H\times\Sigma_n/\Gamma_T)\times_{\Sigma_n} X^{\times n}\big)\to \ccD(V)_n\times_{\Sigma_n} X^n\to X.
\]
The maps in the theorem are the adjoint.
\end{proof}

Using Corollary~\ref{cor:EmbModule}, we get a kind of module structure on a $\ccD(V)$-algebra.
\begin{corollary}\label{cor:Module}
If $X$ is a $\ccD(V)$-module and $T$ is a finite $G$-set such that $T^{G}=\emptyset$, then $\Emb^{G}(T,V)$ also parameterizes maps
\[
\Map(T,X)\times X\to X.
\]
\end{corollary}

There are obvious coherence questions which arise here, and the answer is encoded in Elmendorf's Theorem allowing us to rebuild the homotopy type of a $G$-space out of the diagram of fixed points. This gives us a way to understand the homotopy type of $\ccD(V)_n\times_{\Sigma_n} X^{\times n}$ out of the data of Lemma~\ref{lem:GraphFixedPoints}.

\begin{definition}
Let $\Orb^{G}$ denote the orbit category of $G$. 
Let 
\[
J\colon \Orb^{G}\to\Top^{G}
\]
denote the natural embedding which sends an orbit $G/H$ to $G/H$ viewed as a $G$-space.
If $X$ is a $G$-space, let
\[
X^{(-)}\colon (\Orb^{G})^{op}\to \Top
\]
denote the functor which sends $G/H$ to $\Map^{G}(G/H,X)\cong X^{H}$.
\end{definition}

We will also take advantage of a simplification from the freeness of the $\Sigma_{n}$-action on $\ccD(V)_{n}$. This gives a natural family of subgroups and hence subcategory of the orbit category.

\begin{definition}
If $V$ is an orthogonal representation of $G$, let $\cF^{n}_{V}$ denote the family of subgroups $\Gamma\subset G\times\Sigma_{n}$ such that $\ccD(V)^{\Gamma}_{n}$ is non-empty. Let $\cF^{n}_{V}$ also denote the corresponding sieve in the orbit category.
\end{definition}

\begin{theorem}\label{thm:AnalysisofFreeDisks}
If $X$ is a $G$-CW complex, then we have a natural weak equivalence
\[
\big|B_{\bullet}\big(\ccD(V)_{n}^{(-)}, \cF^{n}_{V}, J|_{\cF^{n}_{V}}\times_{\Sigma_{n}} X^{n}\big)\big|\simeq \ccD(V)_{n}\times_{\Sigma_{n}}X^{n},
\]
where the left-hand side is the geometric realization of the $2$-sided bar construction.
\end{theorem}
\begin{proof}
Elmendorf's Theorem \cite{ElmendorfTheorem} shows that we have a natural $G\times\Sigma_{n}$-weak equivalence
\[
\big|B_{{\bullet}}\big(\ccD(V)_{n}^{(-)}, \Orb^{G\times\Sigma_{n}}, J\big)\big|\simeq \ccD(V)_{n}.
\]
Since $\ccD(V)_{n}^{\Gamma}$ is by definition non-empty only for those $\Gamma\in\cF^{n}_{V}$, and since $\cF^{n}_{V}$ is a sieve, we have a simplicial homeomorphism of simplicial spaces
\[
B_{\bullet}\big(\ccD(V)_{n}^{(-)}, \Orb^{G\times\Sigma_{n}},J\big)\cong B_{\bullet}\big(\ccD(V)_{n}^{(-)}, \cF^{n}_{V}, J\big). 
\]
The result follows from passing the product with $X$ past the geometric realization and commuting the $\Sigma_{n}$-orbits passed the geometric realization. 
\end{proof}

\begin{remark}
The pieces in the two-sided bar construction are ones already encountered in Theorem~\ref{thm:DiskAlgebras}: for some $\Gamma_{T}\in\cF_{V}^{n}$, Lemma~\ref{lem:GraphFixedPoints} identifies $\ccD(V)_{n}^{\Gamma_{T}}$ and 
\[
\big(G\times\Sigma_{n}/\Gamma_{T}\big)\times_{\Sigma_{n}}X^{n}\simeq \Map(T,X).
\]
Theorem~\ref{thm:AnalysisofFreeDisks} shows that the free $\ccD(V)$-algebra on $X$ is homotopically built out of exactly this data.
\end{remark}

\begin{corollary}
Let $K\subset H\subset G$, and let $T$ be a finite $H$-set and $T'$ a finite $K$-set such that $i_K^\ast T\cong T'$. Then the inclusion
\[
\Map^H(T,V)\hookrightarrow \Map^K(T',V)
\]
describes the restriction map on the operadic multiplications.
\end{corollary} 

\subsubsection{Basic consequences}

We first make a simple observation which olds for all algebras over all little disks operads (including for the absurd case of $V=0$).

\begin{proposition}\label{prop:Pointed}
For any $V$ and for any $\ccD(V)$-algebra $X$, we have a canonical $G$-fixed basepoint $e\in X$.
\end{proposition}
\begin{proof}
This is the zeroth structure map corresponding to $\ccD(V)_0=\ast$.
\end{proof}

This connects with traditional constructions in the expected way: trivial subspaces gives increasingly coherently commutative multiplications.

\begin{proposition}\label{prop:Multiplications}
For any $V$ and for any $\ccD(V)$-algebra $X$, if $H\subset G$ has $\dim V^H\geq 1$, then $i_H^\ast X$ has an $E_{\dim V^H}$-structure. The basepoint $e\in i_H^\ast X$ is the unit.
\end{proposition}

\begin{corollary}
If $V$ is any orthongal representation which contains infinitely many copies of the trivial representation, then any $\ccD(V)$-algebra has an underlying $E_\infty$ multiplication.
\end{corollary}

The analysis of the embedding spaces automatically gives us additional structure maps.

\begin{proposition}\label{prop:OperadicNorms}
If $V$ has $\mathring{V}^H\neq \emptyset$, then we have norm maps
\[
\Map_H(G,i_H^\ast X)\cong \Map(G/H,X)\to X,
\]
and if $\pi_{0}\mathring{V}^{H}=\ast$, then this is unique up to homotopy.
\end{proposition}

\begin{remark}
If $V$ is a $G$-universe (so $\ccD(V)$ is an $\Ninfty$ operad), then co\"{i}nduction is both the right and the left adjoint to the forgetful functor. The norm maps realize the counit of the adjunction (with co\"{i}nduction as the left adjoint).
\end{remark}

\begin{proposition}
If $V$ has $\mathring{V}^{H}\neq\emptyset$, then we have an action map
\[
\mu_{G/H}\colon \Map(G/H,X)\times X\to X
\]
making $X$ into a module over the $E_{\dim V^{H}}$-space $\Map(G/H,X)$. If $\pi_{0}\mathring{V}^{H}=\ast$, then this is homotopically unique.
\end{proposition}
\begin{proof}
Proposition~\ref{prop:Multiplications} shows that $i_{H}^{\ast}X$ is an $E_{\dim V^{H}}$-space, and since coinduction is a strong symmetric monoidal functor from $H$-spaces to $G$-spaces, we deduce that
\[
\Map_{H}(G,i_{H}^{\ast}X)\cong \Map(G/H,X)
\]
is an $E_{\dim V^{H}}$-space. The action map is given by any point in the space $\ccD(V)_{|G/H|+1}^{\Gamma_{G/H}}$, which Corollary~\ref{cor:EmbModule} shows is non-empty. 

The only part of the proposition which requires proof is the assertion that this action map makes $X$ a module. For this, note that $\mu_{G/H}$ determines a disk $d_{eH}$ inside of $D(V)$, namely the one corresponding to the coset $eH$. This in turn gives a map
\[
\ccD(V)_{m}\xrightarrow{\Delta^{G}} \Map\big(G/H,\ccD(V)_{m}\big)\xrightarrow{\mu_{G/H}} \ccD(V)_{m|G/H|}
\]
which is $G\times \Sigma_{m}$-equivariant. Here $\Delta^{G}$ is the twisted diagonal adjoint to the identity map on $i_{H}^{\ast}\ccD(V)_{m}$, and $\Sigma_{m}$ acts diagonally in the coinduced space and via the diagonal copy of $\Sigma_{m}^{|G/H|}$ in $\Sigma_{m|G/H|}$. Moreover, we here identify the element $\mu_{G/H}$ with the equivariant map
\[
G\times\Sigma_{|G/H|}/\Gamma_{G/H}\to \ccD(V)_{|G/H|}.
\]

This map views $\ccD(V)_{m}$ as being an arrangement of $m$ little disks in $d_{eH}$, forming the $|G/H|$ collections of little $m$ disks by conjugating by $g$, and then using the natural embedding of $G\times_{H} d_{eH}$ into $V$ given by $\mu_{G/H}$. This map is visibly compatible with compositions in the source, which gives the module structure.

The homotopical uniqueness follows from the homotopical uniqueness of $\mu_{G/H}$ in the case the space $\mathring{V}^{H}$ is path connected.
%
\end{proof}

\begin{remark}
We have not used the full-strength of the operadic action here. The operad $\ccD(i_{H}^{\ast}V)$ acts on $i_{H}^{\ast}X$ and hence on $\Map(G/H,X)$. This also has norm maps, by this procedure, and that endows $X$ also with a compatible module structure over this via
\[
\Map_{H}\big(G,\Map_{K}(H,i_{K}^{\ast}X)\big)\times \Map(G/H,X)\times X\to \Map(G/H,X)\times X\to X.
\]
\end{remark}

\subsection{The case of \texorpdfstring{$\ccD(\sigma)$}{sign disks}}\label{sec:SignedLoopsTopology}
Now let $X$ be an algebra in spaces over $\ccD(\sigma)$. Here, however, we run into the issue that $\sigma$ is one dimensional. In particular, $i_{e}^{\ast}\ccD(\sigma)$ is an $E_{1}$-operad, and hence we have no reason to believe that it is homotopy commutative. One of the surprising features of a $\ccD(\sigma)$-algebra is that it comes equipped with a canonical isomorphism
\[
i_{e}^{\ast} X\xrightarrow{\cong} i_{e}^{\ast} X^{op}.
\]
This is one of the defining features of a $\ccD(\sigma)$-algebra, and it was independently observed by Hahn-Shi in their study of the Real orientation of the Lubin-Tate spectra \cite{HahnShi}.

\subsubsection{Underlying $E_{1}$-space}
Since $\sigma$ is one dimensional, the space $i_{e}^{\ast}X$ is an $E_{1}$-space. The following is classical; we include it to emphasize the equivariance.
\begin{proposition}\label{prop:UnderlyingDsigma}
Pulling the images of points in an embedding of $\{1,\dots, 2k\}$ to the points $-k,\dots, -1, 1, \dots, k$ in $\sigma$ gives a $\Sigma_{2k}$-homotopy equivalence
\[
\ccD(i_e^{\ast}\sigma)_{2k}\simeq \Sigma_{2k}.
\]
Pulling the images of points in an embedding of $\{1,\dots, 2k+1\}$ to the points $-k,\dots, k$ in $\sigma$ gives a $\Sigma_{2k+1}$-homotopy equivalence
\[
\ccD(i_{e}^{\ast}\sigma)_{2k+1}\simeq\Sigma_{2k+1}.
\]
\end{proposition}
We write these symmetrically because then the underlying action of $C_{2}$ on $\sigma$ preserves the preferred images: the sets $\{\pm 1,\dots, \pm k\}$ and $\{\pm 1,\dots, \pm k,0\}$ are $C_{2}$-equivariant subsets of $\sigma$. Of course, the underlying story does not a priori care about what collection of points we choose, and our analysis of the underlying structure does not care what collection of points in $\sigma$ we choose. By having a $C_{2}$-invariant subset, however, we underscore the connection with the fixed points.

Under any choice of points, we see that the underlying $C_{2}$ sign action turns an ordered collection of points into one with the opposite order. Since $C_{2}$ is abelian, this is also a $C_{2}$-equivariant map from $\ccD(\sigma)$ to itself. This shows the following.
\begin{proposition}\label{prop:AntiAut}
If $X$ is a $\ccD(\sigma)$-algebra, then the non-trivial element of $C_{2}$ acts on $X$ as an anti-automorphism of $E_{1}$-algebras.
\end{proposition}
This is exactly what makes the fixed points of a $\ccD(\sigma)$-algebra not have a product: the action and multiplication do not commute.

\begin{remark}
If $X=\Omega^{\sigma} Y$, then the effect of the anti-automorphism from Proposition~\ref{prop:AntiAut} in homotopy is readily understood: it is inversion.
\end{remark}

\subsubsection{Norms and actions}
The analysis in the previous section shows that $X$ comes equipped with a collection of structure maps
\[
\Map(kC_2+\epsilon C_2/C_2,X)\cong \Map(C_{2},X)^{k}\times X^{\epsilon}\to X,
\]
as $k$ ranges over the natural numbers and $\epsilon$ is $0$ or $1$. Since $\dim i_{e}^{\ast}\sigma=1$, this is homotopically quite simple.

\begin{proposition}\label{prop:FixedDsigma}
For all $k$ and $\epsilon$, the space of structure maps is a homotopy discrete $\Aut^{C_{2}}(kC_{2})$-torsor.
\[
\Emb(kC_{2}, \sigma)\simeq \Aut^{C_{2}}(kC_{2}).
\]
\end{proposition}
\begin{proof}
The space $(\mathring{\sigma})^{e}$ is $\mathbb R^{\times}\cong C_{2}\times \mathbb R$. Given any embedding, we can pull the images like beads on a string so that they become $\{\pm 1, \dots, \pm k\}$. This is obviously $\Aut^{C_{2}}(kC_{2})$-equivariant.
\end{proof}
Thus up to homotopy, there is a unique such map for each automorphism of $kC_{2}$, and moreover, the automorphisms generated by $2$ kinds:
\begin{enumerate}
\item Ordinary permutations of the $C_{2}$-sets and
\item the action of $C_{2}=\Aut^{C_{2}}(C_{2})$ on each summand.
\end{enumerate}

Let $\gamma$ be the non-trivial element of $C_{2}$. Here, we have $2$ norm maps
\[
n_{e}^{C_{2}}\colon \Map(C_2,X)\to X\text{ and }n_{e}^{C_{2}}\circ\gamma\colon\Map(C_2,X)\to X,
\]
and $2$ action maps
\[
\mu_{C_{2}}\colon \Map(C_{2},X)\times X\to X\text{ and }\mu_{C_{2}}\circ\gamma\colon \Map(C_{2},X)\times X\to X.
\]

The identification of the embedding space from Proposition~\ref{prop:FixedDsigma} is compatible with restriction, giving us the chosen points in Proposition~\ref{prop:UnderlyingDsigma}. Observing that the ordering of the negative integers is again opposite those of the positive shows the following.

\begin{proposition}
As $E_{1}$-space, we have 
\[
i_{e}^{\ast}\Map(C_{2},X)\cong i_{e}^{\ast}X\times i_{e}^{\ast}X^{op},
\]
the enveloping algebra of $i_{e}^{\ast}X$.
\end{proposition}
Of course, Proposition~\ref{prop:AntiAut} shows that as $E_{1}$-spaces, this is isomorphic to the Cartesian square of $i_{e}^{\ast}X$. Using this hides the more natural ordering of the coordinates.

\begin{corollary}
The restriction of the norm map is the multiplication map
\[
(i_{e}^{\ast}X\times i_{e}^{\ast}X^{op})\to i_{e}^{\ast}X.
\]
The restriction of the action map is the bimodule action map
\[
(i_{e}^{\ast}X\times i_{e}^{\ast} X^{op})\times i_{e}^{\ast} X\to i_{e}^{\ast}X.
\]
The $\Map(C_{2},X)$-module map condition restricts to the observation that these are maps of $i_{e}^{\ast}X$-bimodules.
\end{corollary}

\section{Cohomology of \texorpdfstring{$\ccD(\sigma)$-}{signed loop }algebras}
\subsection{$RO(C_{2})$-graded algebra}
To describe the structure seen in the homology of a $\ccD(\sigma)$-algebra, we need to work in the category of $RO(C_{2})$-graded Mackey functors. A wonderful introduction is provided by Lewis and Mandell in their study of equivariant K\"unneth and Universal Coefficients Spectral Sequences \cite{LewisMandell}, especially Sections 2 and 3 therein. We include here only what we need.

\begin{definition}[{\cite[Definition 2.2]{LewisMandell}}]
An $RO(C_{2})$-graded Mackey functor is a collection of Mackey functors $\mM_{\alpha}$, where $\alpha$ ranges over the virtual representations of $C_{2}$.

A map of $RO(C_{2})$-graded Mackey functors $f\colon\mM_{\star}\to\m{N}_{\star}$ is a collection of maps of Mackey functor $f_{\alpha}\colon \mM_{\alpha}\to\m{N}_{\alpha}$, where $\alpha$ ranges over the virtual representations of $C_{2}$.
\end{definition}

There are natural suspension functors which change the gradings in the expected ways.

\begin{definition}
If $\mM_{\star}$ is an $RO(C_{2})$-graded Mackey functor and if $\alpha$ is a virtual representation of $C_{2}$, then let $\Sigma^{\alpha}\mM_{\star}$ be the $RO(C_{2})$-graded Mackey with
\[
(\Sigma^{\alpha}\mM)_{\tau}=\mM_{\tau-\alpha}.
\]
\end{definition}

Just as with ordinary abelian groups, the category of $RO(C_{2})$-graded Mackey functors inherits a closed symmetric monoidal structure from Mackey functors.

\begin{proposition}[{\cite[Proposition 2.5]{LewisMandell}}]
There is a close symmetric monoidal category structure on the category of $RO(C_{2})$-graded Mackey functors extending the box product on Mackey functors.
\end{proposition}

\begin{definition}
An {\defemph{$RO(C_{2})$-graded Green functor}} is an associative monoid for the box product in $RO(C_{2})$-graded Mackey functors.
\end{definition}

\begin{proposition}[{\cite[Proposition 3.10(a)]{LewisMandell}}]
If $\m{R}_{\star}$ is a commutative $RO(C_{2})$-graded Green functor, then there is closed, symmetric monoidal category of $\m{R}_{\star}$-modules.
\end{proposition}

\begin{remark}
There is a subtlety as to what graded commutativity means in the $RO(C_{2})$-graded setting, since commuting the sign representation past itself introduces an exotic element of the Burnside ring. In the applications below, this unit of the Burnside ring maps to $-1$, and hence this is ordinary commutativity.
\end{remark}

We can further unpack the internal Hom to understand maps in this category a little better. We introduce some useful notation.

\begin{definition}
Let $T$ be a finite $G$-set, and let $\mA^{T}$ denote the Mackey functor:
\[
\mA^{T}(T'):=\mA(T\times T').
\]
Viewing this as an $RO(C_{2})$-graded Mackey functor in degree zero, if $\m{R}_{\star}$ is a commutative $RO(C_{2})$-graded Green functor, let
\[
\mR_{\star}^{T}:=\m{R}_{\star}\Box\mA^{T}.
\]
Finally, if $\mM_{\star}$ is an $\mR_{\star}$-module, then let
\[
\mM_{\star}^{T}:=\mM_{\star}\Box \mR_{\star}^{T}.
\]
\end{definition}

\begin{proposition}[{\cite[Proposition 4.2]{LewisMandell}}]
Let $\mM_{\star}$ be an $\m{R}_{\star}$-module. Then we have a natural isomorphism
\[
\m{\Hom}_{\m{R}_{\star}}^{0}(\Sigma^{\alpha}\m{R}_{\star}^{T},\mM_{\star})(C_{2}/C_{2})\cong \mM_{\alpha}(T).
\]
\end{proposition}
In other words, maps out of the projective $\m{R}$-modules $\Sigma^{\alpha} \m{R}_{\star}^{T}$ recover the value at $T$ of the $\alpha$th Mackey functor of an $\mR_{\star}$-module.

The shifts $\mR_{\star}^{T}$ have a second nice interaction with the symmetric monoidal structure.
\begin{proposition}\label{prop:ProductsofInduceds}
For any finite $G$-sets $T$ and $T'$ and for any $\m{R}_{\star}$-module $\mM_{\star}$, we have natural isomorphisms
\[
\mR_{\star}^{T}\Box_{\mR_{\star}}\mM_{\star}^{T'}\cong\mM_{\star}^{T\times T'}.
\]
\end{proposition}

In our cases of interest below, we will be considering only special modules.

\begin{definition}
Let $\m{R}_{\star}$ be an $RO(C_{2})$-graded Green functor, and let $\mM_{\star}$ be an $\m{R}_{\star}$-module. Then we say that $\mM_{\star}$ is free if there is an isomorphism of $\m{R}_{\star}$-modules
\[
\mM_{\star}\cong\left(\bigoplus_{s\in S_{*}}\Sigma^{\alpha_{s}}\mR_{\star}\right)\oplus\left(\bigoplus_{t\in S_{C_{2}}}\Sigma^{\beta_{t}}\mR_{\star}^{C_{2}}\right).
\]
\end{definition}

\subsection{Coinduction and the norm}
Any algebra $X$ over $D(\sigma)$ in spaces is endowed with operadic transfer maps
\[
\Map_e(C_2,X)\to X.
\]
By naturality of homology, this gives us a map
\[
\m{H}_\star \big(\Map_e(C_2,X);\mM\big)\to \m{H}_{\star} (X;\mM)
\]
for any Mackey functor $\mM$, which is a kind of twisted analogue of the Pontryagin product on the homology of an associative algebra in spaces. We make this precise here.

\begin{definition}
A {\defemph{graded set}} is a set $S$ together with a map $S\xrightarrow{deg} \mathbb Z$. 
\end{definition}

%

This allows us to describe what we see in algebra for the Bredon homology of coinduction. We remark that there is almost the expected universal property: the difficulty is that maps out of $\m{R}$ itself in degree $0$ is a priori non-zero for a great many $RO(C_2)$-graded suspensions of $\m{R}$.

\begin{definition}
Let $S$ be a graded set. We define a kind of $\m{RO}(C_2)$-grading on the $C_2$-set $\Map_e(C_2,S)$ as follows. If $f\in \Map_e(C_2,S)$, then define
\[
\m{\deg}(f):=\begin{cases}
\deg\big(f(e)\big)+ \deg\big(f(g)\big) & f(e)\neq f(g) \\
\deg\big(f(e)\big)\rho_2 & f(e)=f(g).
\end{cases}
\]
\end{definition}

In particular, for for each point $f\in \Map_e(C_2,S)$, we have assigned a virtual representation for the subgroup $\Stab(f)$. Since $C_2$ is abelian, this is all we need to form wedges of the form
\[
\bigvee_{f\in\Map_e(C_2,S)} S^{\m{\deg}(f)}\in\Sp^{C_2}.
\]

\begin{theorem}\label{thm:TwistedKunneth}
Let $R$ be a ring and let $X$ be a space such that the homology of $X$ with coefficients in $R$ is free on a graded set $S$. Then the Bredon homology of $\Map_e(C_2,X)$ with coefficients in $N_e^{C_2} R$ is free on the $\m{RO}(C_2)$-graded set $\Map_e(C_2,S)$.
\end{theorem}

\begin{proof}
The graded set $S$ gives a weak equivalence of $HR$-module spectra
\[
\bigvee_{s\in S} S^{\deg(s)}\wedge HR\xrightarrow{\simeq} \Sigma^{\infty}_{+}X\wedge HR.
\]
If we apply the norm $N_e^{C_2}$ to both sides, then we deduce an equivalence of $N_e^{C_2} HR$-module spectra
\[
N_e^{C_2}\left(\bigvee_{s\in S} S^{\deg(s)}\right)\wedge N_e^{C_2}HR\xrightarrow{\simeq} N_e^{C_2} (\Sigma^{\infty}_{+}X)\wedge N_e^{C_2}HR.
\]
The distributive law applied to the source gives an isomorphism of $N_e^{C_2} HR$-modules
\[
N_e^{C_2}\left(\bigvee_{s\in S} S^{\deg(s)}\right)\wedge N_{e}^{C_2}HR\simeq \bigvee_{f\in\Map_e(C_2,S)} S^{\m{\deg}(f)} \wedge N_e^{C_2} HR.
\]
Similarly, since the infinite suspension is strong $G$-symmetric monoidal, we have a natural equivalence
\[
N_e^{C_2} \Sigma^{\infty}_+ X\simeq \Sigma^{\infty} \Map_e(C_2,X).
\]
Finally, since $HR$ can be modeled by a commutative ring spectrum, $N_e^{C_2}HR$ is a $C_2$-equivariant commutative ring spectrum, so the category of modules thereover is symmetric monoidal. The map from a $(-1)$-connected $C_2$-equivariant commutative ring spectrum to its zeroth Postnikov section preserves $C_2$-equivariant commutative ring spectra, and 
\[
\m{\pi}_0 N_e^{C_2} HR\cong N_e^{C_2} R,
\]
where the righthand norm is the Mazur-Hoyer norm on Mackey functors \cite{MazurArxiv}, \cite{HoyerThesis}. Base-changing along the zeroth Postnikov map
\[
N_e^{C_2} HR\to H N_e^{C_2} R
\]
gives the desired result.
\end{proof}

\begin{remark}
There is also a coordinate-free version of this result. We have a canonical isomorphism of simplicial $G$-sets
\[
\Map_e\big({C_2},\Sing_\bullet(X)\big)\cong \Sing_{\bullet}\big(\Map_e(C_2,X)\big).
\]
In particular, this produces a canonical isomorphism of simplicial Mackey functors upon applying the Burnside Mackey functor (or more generally, any other coefficients). The left-hand side is essentially the definition of the Mazur-Hoyer norm is simplicial Mackey functors. We will return to this more generally in a subsequent paper.
\end{remark}

We pause here to clarify a small notational point. The set over which we take our wedge in the proof of Theorem~\ref{thm:TwistedKunneth} is a $C_{2}$-set. In particular, the action on the indexing set is combined with the action on the individual factors. Thus if $f\in\Map_{e}(C_{2},S)$ has $f(e)\neq f(g)$, then the summands corresponding to $f$ and to $g\cdot f$ are switched by the group action. Thus we could rewrite the sum as
\[
\bigvee_{f\in\Map_{e}(C_{2},S)/C_{2}} C_{2+}\wedge_{Stab(f)} S^{\m{\deg}(f)}\wedge HN_{e}^{C_{2}}R.
\]

Since cohomology with coefficients in a field is always free, this gives us a nice family of space for which we know Bredon homology groups.

\begin{definition}
Let $\m{B}=N_{e}^{C_{2}}\mathbb F_{2}$. This is the Green functor
\[
\xymatrix{
{\Z/4}
	\ar@(l,l)[d]_{1}
	\\
{\Z/2}
	\ar@(r,r)[u]_{2}
}
\]
\end{definition}

\begin{definition}
Let 
\[
\mB_{\star}:=\m{\pi}_{\star}H\mB.
\]
\end{definition}

\begin{proposition}
The $RO(C_{2})$-graded Mackey functor $\m{B}$ is naturally a commutative $RO(C_{2})$-graded Green functor.
\end{proposition}
\begin{proof}
The {\EM} spectrum $H\mB$ is a commutative ring spectrum by work of Ullman. Since $\m{\pi}_{\star}$ is a lax monoidal functor (see, for a complete proof \cite[Appendix A]{LewisMandell}), the result follows.
\end{proof}
Since $\m{B}$ is a quotient of the constant Mackey functor $\Z$, the sign rule here is the ordinary one, using only the underlying dimension of the representations. 

\begin{corollary}
If $X$ is any space, then the Bredon homology of $\Map_e(C_2,X)$ with coefficients in $\m{B}$ is a free $\mB_{\star}$-module.
\end{corollary}

This can also be identified with a purely algebraic functor: the norm in $RO(C_{2})$-graded Mackey functors. This is the obvious extension of the Mazur-Hoyer norm, taking into consideration the canonical isomorphism
\[
N_{e}^{C_{2}} \Sigma^{k}M\cong \Sigma^{k\rho} N_{e}^{C_{2}}M.
\]
This, plus the distributive law, immediately gives the following.

\begin{corollary}\label{cor:HomologyofCoinduced}
If $X$ is any space, then the Bredon homology of $\Map_e(C_2,X)$ with coefficients in $\m{B}$ is isomorphic to
\[
N_{e}^{C_2} H_\ast (i_{e}^{\ast}X;\mathbb F_2).
\]
\end{corollary}

More generally, we get a freeness result for any module over $\m{B}$. Since the Bredon homology with coefficients in $\m{B}$ is a free $\m{RO}(C_{2})$-graded module, it is flat over $\m{B}_{\star}$. In particular, the universal coefficients spectral sequence collapses.
\begin{corollary}
If $\mM$ is any $\m{B}$-module, then 
\[
\m{H}_{\star}\big(\Map_{e}(C_{2},X);\mM\big)\cong N_{e}^{C_{2}} \big(H_{\ast}(X;\F_{2})\big)\Box_{\m{B}} \mM,
\]
which splits as an $\m{RO}(C_{2})$-graded sum of copies of $\mM$.
\end{corollary}

In particular, this implies a fairly strong form of a K\"unneth isomorphism.

\begin{proposition}\label{prop:HomologyCoindMod}
If $X$ is a space and $Y$ is a $C_{2}$-space, then we have a natural isomorphism
\[
\m{H}_{\star}\big(\Map_{e}(C_{2},X)\times Y;\m{B}\big)\cong \m{H}_{\star}\big(\Map_{e}(C_{2},X);\m{B}\big)\Box_{\m{B}_{\star}}\m{H}_{\star}(Y;\m{B}).
\]
\end{proposition}
\begin{proof}
The Lewis-Mandell $RO(C_{2})$-graded K\"unneth spectral sequence collapses, since $\m{H}_{\star}\big(\Map_{e}(C_{2},X);\m{B}\big)$ is a free module.
\end{proof}

Finally, there is a pointed version of all of these results; the proofs are essentially unchanged.

\begin{proposition}\label{prop:RedHomologyNorm}
If $X$ is a pointed space, then 
\[
\m{\tilde{H}}_{\star}\big(N_{e}^{C_{2}}X;\m{B}\big)\cong N_{e}^{C_{2}} \big(\tilde{H}_{\ast}(X;\mathbb F_{2})\big).
\]
\end{proposition}

%
%
%
%

\subsection{Cohomology of $\ccD(\sigma)$-spaces}
We can now combine the structural results from \S~\ref{sec:SignedLoopsTopology} to determine the structure present in homology of a signed loop space. We begin with a small definition.

\begin{definition}
Let $R$ be an associative ring. A {\defemph{pointed}} $R$-module is an $R$-module $M$ together with an element $m\in M$.
\end{definition}
By the free-forget adjunction, this is of course the same data as an $R$-module homomorphism $R\to M$.

\begin{theorem}\label{thm:HstarStructure}
Let $X$ is an algebra over $\ccD(\sigma)$. Let 
\[
R_{\ast}=H_\ast(i_e^{\ast} X;\mathbb F_2)
\]
and let 
\[
\m{R}_{\star}=N_e^{C_2} R_{\ast}\cong \m{H}_{\star}\big(\Map(C_{2},X);\mB\big).
\]
Then
\begin{enumerate}
\item $R_\ast$ is a graded, associative ring equipped with an isomorphism 
\[
\overline{(\mhyphen)}\colon R_\ast\to R_\ast^{op}
\] 
\item $\m{H}_\star(X;\mB)$ is naturally a pointed $\m{R}_{\star}$-module, and the restriction of distinguished point is the multiplicative unit in $R_\ast$.
\end{enumerate}
\end{theorem}
%
%

\begin{remark}
The pointing in Theorem~\ref{thm:HstarStructure} gives us a kind of norm map on the homology. This has been used most recently by Behrens-Wilson in their construction of $H\m{\F}_2$ as a Thom spectrum \cite{BehWil}.
\end{remark}

\subsection{Aside: Even more structure}
Just as the homology of any space with coefficients in a field is a co-commutative co-algebra, the homology of a $C_{2}$-space with coefficients in $\mB$ has kinds of comultiplications. These arise from the two [twisted] diagonals:
\[
X\xrightarrow{\Delta} X\times X\text{ and } X\xrightarrow{\Delta_{g}} \Map(C_{2},X).
\]
Applying homology with coefficients in $\mB$ then gives structure maps.

\begin{proposition}\label{prop:CoNorm}
If $X$ is any $C_{2}$-space, then the homology of $X$ with coefficients in $\mB$ has a natural co-norm map
\[
\m{H}_{\star}(X;\mB)\to N_{e}^{C_{2}} H_{\ast}(i_{e}^{\ast}X;\F_{2}).
\]
\end{proposition}
\begin{proof}
Corollary~\ref{cor:HomologyofCoinduced} identifies the homology of the target of the map $\m{H}_{\star}\Delta_{g}$.
\end{proof}

In the case we have flatness, then we can deduce also a comultiplication.

\begin{proposition}
If $X$ is a $C_{2}$-space such that $\m{H}_{\star}(X;\mB)$ is a flat $\mB_{\star}$-module, then we have a natural comultiplication
\[
\m{H}_{\star}(X;\mB)\to \m{H}_{\star}(X;\mB)\Box_{\mB_{\star}} \m{H}_{\star}(X;\mB).
\]
In this case, $\mH_{\star}(X;\mB)$ is naturally a co-Tambara functor.
\end{proposition}

Since these are induced by natural maps of spaces, we deduce that all of the structure maps described before are actually maps of co-Tambara functors.

\section{The signed James construction}
\subsection{Construction and interpretation}
Just as classically there is a combinatorial model for the loop space of the suspension of a space, there is an elegant combinatorial model due to Rybicki for the signed loops on the signed suspension of a $C_2$-space.

\begin{definition}[{\cite[Section 2]{RybickiJames}}]
If $X$ is a pointed $C_2$ space, then let 
\[
J^{\sigma}_n(X):=\coprod_{k= 0}^{n} X^{\times k} / \sim,
\]
where $C_2$ acts on $X^{\times k}$ via
\[
(x_1,\dots,x_k)\mapsto (\bar{x}_k,\dots,\bar{x}_1),
\]
and where $\sim$ is the equivalence relation which simply omits any coordinate which is the basepoint.

Let $J^{\sigma}(X)$ be the colimit of $J^{\sigma}_n(X)$ as $n$ varies.
\end{definition}

The intertwining of the $C_2$-action on the space and on the Cartesian factors is what makes this definition viable. 
For us, we will need an inductive pushout description of the finite $J^{\sigma}_k(X)$. 

%

\begin{lemma}\label{lem:PushoutSquares}
Let $X$ be a pointed $C_2$-space. For a finite $C_2$-set $T$, let $F_T(X)$ denote the ``fat wedge'' for $\Map(T,X)$, that is, the collection of points in $\Map(T,X)$ for which one of the coordinates is the basepoint. 

We have pushout squares of $C_2$-equivariant spaces:
\begin{center}
\begin{tikzcd}
F_{nC_2}(X) 
	\ar[r]
	\ar[d]
	&
\Map(C_2,X)^n
	\ar[d]
	\\
J_{2n-1}^{\sigma}(X)
	\ar[r]
	&
J_{2n}^{\sigma}(X)
\end{tikzcd} and
\begin{tikzcd}
F_{nC_2\amalg \ast}(X)
	\ar[d]
	\ar[r]
	&
\Map(C_2,X)^n\times X
	\ar[d]
	\\
J_{2n}^{\sigma}(X)
	\ar[r]
	&
J_{2n+1}^{\sigma}(X).
\end{tikzcd}
\end{center}
\end{lemma}

\subsection{Splitting the signed James construction}
Classically, the James construction splits after a single suspension into a wedge of a suspension of smash powers of the original space. Here, the presence of two kinds of products requires two suspensions.

\begin{theorem}\label{thm:JamesSplitting}
For any pointed $C_2$-space of the homotopy type of a $C_2$-CW complex, we have a natural weak equivalence
\[
\Sigma^{\rho}_{+} J^{\sigma} (X)\simeq\bigvee_{k\geq 0} \Sigma^{\rho}(N_{e}^{C_2} i_{e}^{\ast}X^{\wedge k})\wedge (S^0\vee X).
\]
\end{theorem}
\begin{proof}
This is immediate from Lemma~\ref{lem:PushoutSquares}, once we remember that we have equivariant homeomorphisms
\[
\Map(C_2,X^{\times k})/F_{kC_2}(X)\cong (N_{e}^{C_2} i_{e}^{\ast} X)^{\wedge k},
\]
and
\[
\big(\Map(C_2,X^{\times k})\times X\big)/F_{kC_2\amalg\ast}(X)\cong \big(N_e^{C_2}i_e^{\ast}X^{\wedge k}\big)\wedge X.
\]

Theorem~\ref{thm:SplittingCoinduction} below shows that upon a single sign or regular suspension, the top row in both pushout squares from Lemma~\ref{lem:PushoutSquares} splits. In particular we deduce two splittings
\[
\Sigma^{\rho} J_{2n}^{\sigma}(X)\simeq \Sigma^{\rho} J_{2n-1}^{\sigma}(X)\wedge N_{e}^{C_2}(X)^{\wedge n}
\]
and
\[
\Sigma^{\rho} J_{2n+1}^{\sigma}(X)\simeq\Sigma^{\rho} J_{2n}^{\sigma}(X)\wedge \big(N_{e}^{C_2}(X)^{\wedge n}\big)\wedge X.
\]
Splicing these together gives the desired result.
\end{proof}

\begin{corollary}
For any $k$ and for any $C_2$-space $X$, we have James-Hopf maps of the form
\[
h_{kC_2}\colon \Omega^{\sigma}\Sigma^{\sigma} X\to \Omega^{\rho}\Sigma^{\rho} N_{e}^{C_2} i_{e}^{\ast}X^{\wedge k}
\]
and
\[
h_{kC_2+1}\colon\Omega^{\sigma}\Sigma^{\sigma} X\to\Omega^{\rho}\Sigma^{\rho} (X\wedge N_{e}^{C_2} i_{e}^{\ast} X^{\wedge k}).
\]
\end{corollary}
\begin{proof}
These are adjoint to the projections onto the summands of the splitting of 
\[
\Sigma^{\rho} J^{\sigma}(X)\simeq\Sigma^{\rho}\Omega^{\sigma}\Sigma^{\sigma}X.\qedhere
\]
\end{proof}

Specializing to the case of $X$ a sphere, we have a sequence of maps which are equivariant refinements of the James-Hopf map.
\begin{corollary}
For any natural numbers $j,k$, we have a James-Hopf map
\[
h_{C_2}\colon \Omega^{\sigma} S^{j+(k+1)\sigma}\to \Omega^{\rho} S^{(j+k+1)\rho}.
\]
\end{corollary}

There are two somewhat surprising features (and computationally vexing) aspects of this:
\begin{enumerate}
\item We see an extra loops appearing (here in the form of $\Omega^{\rho}=\Omega\Omega^{\sigma}$), making an analysis of the fiber trickier than in the classical case, and
\item the target of the map depends only on the underlying, non-equivariant sphere, and not on the particular equivariant sphere used.
\end{enumerate}

\subsection{The homology of the James construction}

As an immediate consequence of the stable splitting of the James construction, we can compute the Bredon homology of the signed James construction with coefficients in $\mB$.

\begin{theorem}\label{thm:HomologyofJsigma}
For any $C_{2}$-space $X$, we have an isomorphism of $RO(C_{2})$-graded Mackey functors
\[
\m{H}_{\star}\big(J^{\sigma}X;\mB\big)\cong 
\left(\bigoplus_{i = 0}^{\infty}N_{e}^{C_{2}}\big(\tilde{H}_{\ast}(i_{e}^{\ast}X;\F_{2})^{\otimes i}\big)\right)
\Box \big(\mB_{\star}\oplus \tilde{\m{H}}_{\star}(X;\m{B})\big).
\]
\end{theorem}
\begin{proof}
The stable splitting of the signed James construction yields
\[
\Sigma^{\infty\rho}_{+} J^{\sigma}X\simeq \left(\bigvee_{i=0}^{\infty} N_{e}^{C_{2}}(i_{e}^{\ast}X)^{\wedge i}\right)\wedge (S^{0}\vee X).
\]
Proposition~\ref{prop:RedHomologyNorm} together with the classical K\"unneth theorem shows that 
\[
\m{H}_{\star}\left(\bigvee_{i=0}^{\infty} N_{e}^{C_{2}}(i_{e}^{\ast}X)^{\wedge i};\mB\right)\cong 
\left(\bigoplus_{i = 0}^{\infty}N_{e}^{C_{2}}\big(\tilde{H}_{\ast}(i_{e}^{\ast}X;\F_{2})^{\otimes i}\big)\right).
\]
The pointed version of Proposition~\ref{prop:HomologyCoindMod} then gives the rest.
\end{proof}

\begin{remark}
There is a version of the Bott-Samelson theorem for the signed loops, which allows us to describe the homology as a particular algebraic functor \cite{BottSam}. For this case, it is possible to work out directly, though the analysis is a bit unenlightening. In short, it is the free $\ccD(\sigma)$-algebra in Mackey functors on $\m{\tilde{H}}_{\star}(X;\mB)$. The difficulty here is describing the twisted powers, just as with the norm in spaces above. In joint work with Tyler Lawson, we will return to this point.
\end{remark}

\appendix

\section{Splitting $C_2$-coinduction}
Classically, the Cartesian product splits into the wedge and smash products after a single suspension. Equivariantly, we have instead our $G$-Cartesian monoidal structure on $G$-spaces, and this necessitates a more general splitting result. We begin with a non-example.

\begin{proposition}
All suspensions by trivial representations of the $C_2$-equivariant map $S^{\sigma}\to C_{2+} \wedge S^1$ with cofiber $\Map_e(C_2,S^1)$  are essential.
\end{proposition}
\begin{proof}
The standard picture showing the $2$-torus as a quotient of the square can be visibly made equivariant. In this case, all that follows is simply reinterpreting that picture.

First observe that the standard $C_2$-equivariant cell-structure of $S^{\sigma}$ (and hence of $S^{k+\sigma}$) gives us an exact sequence (of pointed sets, if $k=0$):
\[
[S^{1+k},C_{2+}\wedge S^{1+k}]^{C_2}\to [C_{2+}\wedge S^{1+k},C_{2+}\wedge S^{1+k}]^{C_2}\to [S^{\sigma+k},C_{2+}\wedge S^{1+k}]^{C_2}\to [S^k, C_{2+}\wedge S^{1+k}]^{C_2}.
\]
Since the $C_2$-action on $S^k$ and $S^{1+k}$ is trivial, and since the $C_2$-fixed points of any space of the form $C_{2+}\wedge X$ are a point, we conclude that the two outer terms both vanish. Thus for all $k$, the map
\[
[S^{1+k}, i_e^{\ast} C_{2+}\wedge S^{1+k}]\cong [C_{2+}\wedge S^{1+k},C_{2+}\wedge S^{1+k}]^{C_2}\xrightarrow{q^\ast} [S^{\sigma+k},C_{2+}\wedge S^{1+k}]^{C_2}
\]
induced by the collapse map $S^{\sigma}\to C_{2+}\wedge S^1$ is an isomorphism. In particular, the suspension map is an isomorphism as well for all $k\geq 1$, by the classical result, and the suspension for $k=0$ is abelianization. 

We pause here to point out a simple geometric fact: the inverse to $q^{\ast}$ is the function that assigns to a $C_2$-equivariant map 
\[
f\colon S^{\sigma}\to C_{2+}\wedge S^{1}
\]
the non-equivariant map $f|_{Im(z)\geq 0}$, where we are viewing $S^\sigma$ as the unit circle in $\mathbb C$. Unpacking the standard description of the attaching map for $S^1\times S^1$ shows that the map $S^{\sigma}\to C_{2+}\wedge S^1$ here corresponds to the standard crush map $S^1\to S^1\vee S^1$, which survives abelianization.
\end{proof}
\begin{remark}
The curious fact we used here is that while the underlying non-equivariant map is the commutator of the two canonical inclusions $S^1\hookrightarrow S^1\vee S^1$, as a $C_2$-equivariant map, the upper and lower semi-circles are oriented the same way. In particular, we never trace out the inverses, equivariantly.
\end{remark}

As should at this point be unsurprising, if we use twisted suspensions and a twisted version of the join, then we can prove a splitting result.

\begin{definition}
If $X$ is a $C_2$-space, then let
\[
X^{\ast C_2}:= \Map(C_2,X)\times D(\sigma)/\sim,
\]
where $\simeq$ is the equivalence relation
\begin{align*}
&\forall x,y,y'\in X,\quad (x,y,-1)\sim (x,y',-1)\\
&\forall x, x', y\in X,\quad  (x,y,1)\sim (x',y,1) .
\end{align*}

If $X$ is a pointed $C_2$-space with basepoint $x_0$, then the reduced version (which we will also denote $X^{\ast C_2}$) collapses the subspace $(x_0,x_0)\times D(\sigma)$.
\end{definition}

Since the group action on $\Map(C_2,X)\times D(\sigma)$ is given by
\[
(x,y,t)\mapsto (y,x,-t),
\]
the equivalence relation is built to be equivariant. 

Using this twisted version of the join, we can prove the analogue of the ordinary splitting of the product.

\begin{theorem}\label{thm:SplittingCoinduction}
Let $X$ be a based $C_2$ space of the homotopy type of a $C_2$-CW complex. Then we have a natural weak equivalence
\[
\Sigma^{\sigma} \Map(C_2,X)\simeq (C_{2+}\wedge \Sigma X)\vee\Sigma^{\sigma} N_e^{C_2} X.
\]
\end{theorem}
\begin{proof}
Let $x_0$ denote the basepoint of $X$. There is a canonical map 
\[
C_{2+}\wedge X\to X^{\ast C_2}
\]
given by
\[
(g,x)\mapsto g\cdot (x,x_0, 1).
\]
Using this, we can attach the cone on $C_{2+}\wedge X$:
\[
Y=C_{2+}\wedge CX\cup X^{\ast C_2}.
\]
On the one hand, since the cone on $C_{2+}\wedge X$ is equivariantly contractible, the quotient map crushing it to a point is a homotopy equivalence:
\[
Y\xrightarrow{\simeq} Y/(C_{2+}\wedge CX)\cong \Sigma^{\sigma}\Map(C_2,X),
\]
where the final homeomorphism is the standard one identifying $D(\sigma)/S(\sigma)\cong S^{\sigma}$.

On the other hand, the non-equivariant embedding $CX\hookrightarrow X^{\ast C_2}$ given by
\[
(x,t)\mapsto (x,2t-1)
\]
obviously extends to an equivariant embedding 
\[
C_{2+}\wedge CX\hookrightarrow X^{\ast 2},
\]
which shows that the subcomplex $C_{2+}\wedge X$ inside of the twisted join is null-homotopic. Thus we have an equivalence
\[
Y\simeq Y/(C_{2+}\wedge X)\simeq (C_{2+}\wedge \Sigma X)\vee X^{\ast C_2}/(C_{2+}\wedge CX).
\]
However, the embedded copy of $C_{2+}\wedge CX$ consists of all points of $X^{\ast C_2}$ of the form $(x,y,t)$ where at least one of $x$ and $y$ is $x_0$. In particular, we have an obvious homeomorphism
\[
X^{\ast C_2}/(C_{2+}\wedge CX)\cong\Sigma^{\sigma} N_e^{C_2} X,
\]
finishing the proof.
\end{proof}
\begin{remark}
The apparently asymetrical appearance of $\Sigma$ rather than $\Sigma^{\sigma}$ for the wedge can be explained by the fact that these agree when we forget the equivariance. We therefore have a natural homeomorphism
\[
\Sigma^{\sigma} C_{2+}\wedge X\cong\Sigma C_{2+}\wedge X.
\]
\end{remark}

\bibliographystyle{plain}

\bibliography{SignedLoops}

\begin{thebibliography}{10}

\bibitem{BehWil}
Mark Behrens and Dylan Wilson.
\newblock A {$C_2$}-equivariant analog of {M}ahowald's {T}hom spectrum theorem.
\newblock arxiv.org:1707.02582, 2017.

\bibitem{BHNinfty}
Andrew~J. Blumberg and Michael~A. Hill.
\newblock Operadic multiplications in equivariant spectra, norms, and
  transfers.
\newblock {\em Adv. Math.}, 285:658--708, 2015.

\bibitem{BottSam}
R.~Bott and H.~Samelson.
\newblock On the {P}ontryagin product in spaces of paths.
\newblock {\em Comment. Math. Helv.}, 27:320--337 (1954), 1953.

\bibitem{CohenHstar}
Fred Cohen.
\newblock Homology of {$\Omega ^{(n+1)}\Sigma ^{(n+1)}X$} and
  {$C_{(n+1)}X,\,n>0$}.
\newblock {\em Bull. Amer. Math. Soc.}, 79:1236--1241 (1974), 1973.

\bibitem{ElmendorfTheorem}
A.~D. Elmendorf.
\newblock Systems of fixed point sets.
\newblock {\em Trans. Amer. Math. Soc.}, 277(1):275--284, 1983.

\bibitem{HahnShi}
Jeremey Hahn and XiaoLin~Danny Shi.
\newblock Real orientation of {M}orava {$E$}-theories.
\newblock arxiv.org: 1707.03413, 2017.

\bibitem{Hausch}
H.~Hauschild.
\newblock \"aquivariante {K}onfigurationsr\"aume und {A}bbildungsr\"aume.
\newblock In {\em Topology {S}ymposium, {S}iegen 1979 ({P}roc. {S}ympos.,
  {U}niv. {S}iegen, {S}iegen, 1979)}, volume 788 of {\em Lecture Notes in
  Math.}, pages 281--315. Springer, Berlin-New York, 1980.

\bibitem{HoyerThesis}
Rolf Hoyer.
\newblock {\em Two topics in stable homotopy theory}.
\newblock PhD thesis, University of Chicago, 6 2014.

\bibitem{LewisMandell}
L.~Gaunce Lewis, Jr. and Michael~A. Mandell.
\newblock Equivariant universal coefficient and {K}\"unneth spectral sequences.
\newblock {\em Proc. London Math. Soc. (3)}, 92(2):505--544, 2006.

\bibitem{MayGeom}
J.~Peter May.
\newblock {\em Geometry of iterated loop spaces}.
\newblock Lecture Notes in Mathematics, Vol. 271. Springer-Verlag, Berlin-New
  York, 1972.

\bibitem{MazurArxiv}
Kristen Mazur.
\newblock An equivariant tensor product on {M}ackey functors.
\newblock arxiv.org: 1508.04062, 2015.

\bibitem{RouSand}
Colin Rourke and Brian Sanderson.
\newblock Equivariant configuration spaces.
\newblock {\em J. London Math. Soc. (2)}, 62(2):544--552, 2000.

\bibitem{RybickiJames}
S{\l}awomir Rybicki.
\newblock {$Z_2$}-equivariant {J}ames construction.
\newblock {\em Bull. Polish Acad. Sci. Math.}, 39(1-2):83--90, 1991.

\bibitem{SegalConf}
Graeme Segal.
\newblock Configuration-spaces and iterated loop-spaces.
\newblock {\em Invent. Math.}, 21:213--221, 1973.

\end{thebibliography}

\end{document}